\documentclass[10pt]{article}
\usepackage[matrix,arrow,curve]{xy}
\usepackage{amssymb,amsthm}
\usepackage{supertabular}

\usepackage{graphicx}
\usepackage{pstricks}
\usepackage{psfrag}

\usepackage{epic}
\usepackage{eepic}
\newcommand{\ddate}{February 25 2013}

\newtheorem{dummy}{anything}[section]
\newtheorem{Theorem}[dummy]{Theorem}
\newtheorem{Lemma}[dummy]{Lemma}
\newtheorem{Proposition}[dummy]{Proposition}

\newtheorem{Remark}[dummy]{Remark}

\newtheorem*{mainth}{Main theorem}

\newcommand{\bbr}{{\mathbb R}}

\newcommand{\bbc}{{\mathbb C}}
\newcommand{\bbz}{{\mathbb Z}}

\newcommand{\calc}{{\mathcal C}}
\newcommand{\cald}{{\mathcal D}}

\newcommand{\calg}{{\mathcal G}}

\newcommand{\calj}{{\mathcal J}}

\newcommand{\caln}{{\mathcal N}}

\newcommand{\calt}{{\mathcal T}}
\newcommand{\calu}{{\mathcal U}}

\newcommand{\mancqfd}{\hfill \ensuremath{\Box}}

\newcommand{\pcirc}{\kern .7pt {\scriptstyle \circ} \kern 1pt}

\newcommand{\mun}{{-1}}

\newcommand{\dtc}{\dot{\hskip -.12em\tilde c}}

\newcommand{\scr}{\scriptscriptstyle}

\renewcommand{\:}{\colon}

\newcommand{\sk}[1]{\vskip #1 mm}
\newcommand{\eqref}[1]{(\ref{#1})}
\newcommand{\hfl}[2]{\smash{\mathop{\hbox to 1 truecm{\kern %
3pt\rightarrowfill\kern 3pt}}%
\limits^{\scriptstyle#1}_{\scriptstyle#2}}}
\newcommand{\cqfd}{\unskip\kern 6pt\penalty 500%
\raise -2pt\hbox{\vrule\vbox to10pt{\hrule width %
4pt\vfill\hrule}\vrule}\smallskip}
\newcommand{\proref}[1]{Proposition~\ref{#1}}
\newcommand{\remref}[1]{Remark~\ref{#1}}
\newcommand{\lemref}[1]{Lemma~\ref{#1}}

\newcommand{\thref}[1]{Theorem~\ref{#1}}

\newcommand{\secref}[1]{Section~\ref{#1}}

\newcommand{\dfn}[1]{{\it #1}}

\newcommand{\ga}{\Gamma}

\newcommand{\eqncount}{\setcounter{equation}{\value{dummy}}%
\addtocounter{dummy}{1}}

\newcommand{\beq}[1]{\eqncount\begin{equation}\label{#1}}
\newcommand{\eeq}{\end{equation}}

\newcommand{\dia}[1]{\begin{array}{c}{\xymatrix@C-3pt@M+2pt@R-4pt{#1 }}\end{array}}

\newcommand{\tri}{{\rm Tri}}
\newcommand{\ftri}{{\rm Tri}_{fl}}

\newcommand{\bdg}{{\rm bdg}}

\title{Triangles on planar Jordan  $C^1$-curves}
\author{Jean-Claude HAUSMANN}
\date{\ddate}
\begin{document}
\maketitle %\tableofcontents

\begin{abstract}
We prove that a Jordan $\calc^1$-curve in the plane contains the vertices of
any non-flat triangle, up to translation and homothety with positive ratio. This is false if the curve
is not $C^1$. The proof uses a bit configuration spaces, differential and algebraic topology as well as the 
smooth Schoenflies theorem. A partial generalization holds true in higher dimensions.
\end{abstract}

The aim of this note is to prove the main theorem below, in which the following standard definitions
are used. A \dfn{Jordan $\calc^1$-curve $\ga$} is a 
connected closed $\calc^1$-submanifold of the plane. Equivalently, $\ga$ 
is the image of an injective $\calc^1$-immersion of the unit circle $S^1$ into the plane.
A triangle is \dfn{flat} if it is contained in a straight line.

\begin{mainth}
Let $\ga$ be a Jordan $\calc^1$-curve and let $T$ be a non-flat triangle in the plane.
Then, by a translation and a homothety with positive ratio, $T$ may be transformed into a triangle whose
vertices lie on $\Gamma$. 
\end{mainth}

%\iffalse

\psfrag{A}{$\Gamma$}
\psfrag{T}{$T$}
\psfrag{S}{$T'$}
\begin{minipage}{6cm}
\hskip 25mm
\scalebox{0.5}{
\includegraphics{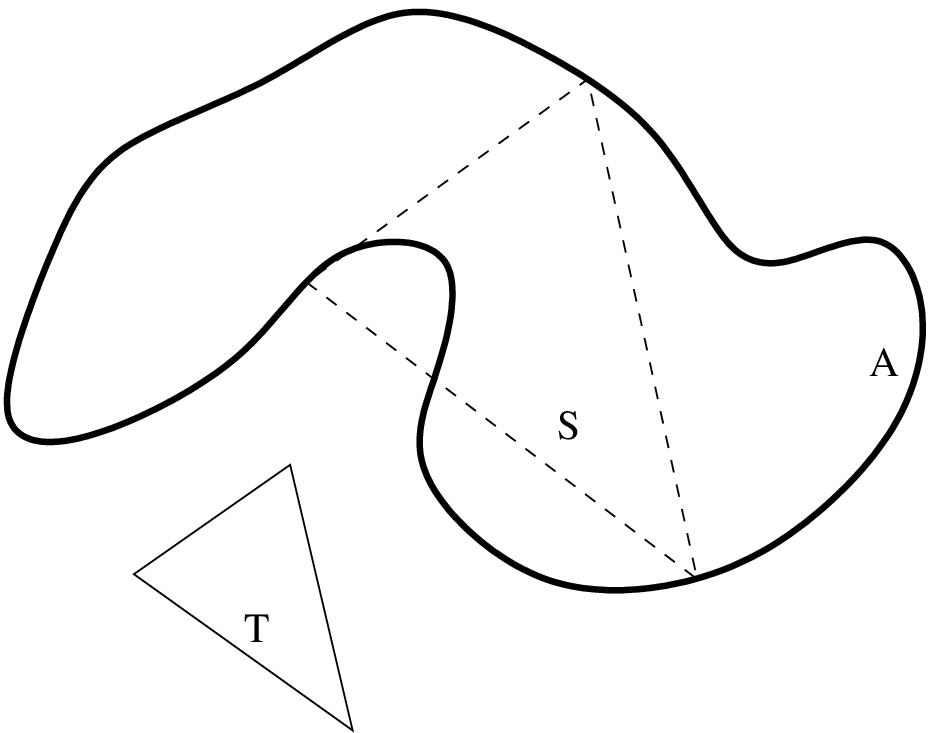}
}
\end{minipage}

\sk{-3}
\noindent
\begin{minipage}{9cm}
The main theorem is wrong if the curve $\ga$ is not of class $\calc^1$, as seen by the example
of a half-lemniscate, with parametrization in polar coordinates given by $r(\theta)=\cos 2\theta$ 
($|\theta|\leq\pi/4$). Let $T$ be an isosceles triangle with a vertical basis $AB$ and the third vertex $C$ on the left of $AB$.
If $T$ is %on~$\Gamma$,
%$\ga$, then $C=0$ by symmetry (since $\ga$ and $T$ are invariant 
% under the reflection through the horizontal axis).
% If the angle at $C$ is obtuse, this is impossible since $\ga$ lies in in the polar domain $|\theta|\leq\pi/4$.
% Note that $\ga$ is of class $\calc^\infty$ except at $0$, where the two tangents form a right angle.   
\end{minipage}
\begin{minipage}{6cm}
%\begin{center}
\sk{19}\hskip 2mm
\setlength{\unitlength}{.2mm}
%\scalebox{1.6}{
\begin{pspicture}(0,-2.1)(0,0)
\scalebox{0.9}{
\psline[linewidth=0.5pt,linestyle=dotted](-0.3,0)(3.5,0)
\psarc[linewidth=0.5pt]{<->}(3.0,0){0.4}{-30}{30}
\put(2.5,0.9){\small $\Gamma$}
\put(1.23,0.2){\small $T$}
\psline[linewidth=0.5pt,linestyle=dashed](1.19,0.76)(1.19,-0.76)(0.77,0)(1.19,0.76)
\scalebox{3}[3.3]{
\psline[linewidth=0.5pt]%
(0.000563312,-0.000562864)(0.025674,-0.0247738)(0.0515901,-0.0480719)(0.0782552,-0.0704087)%
(0.105611,-0.0917388)(0.133596,-0.11202)(0.162148,-0.131212)(0.191201,-0.149279)(0.220691,-0.166189)%
(0.250548,-0.181912)(0.280704,-0.196422)(0.311088,-0.209696)(0.34163,-0.221715)(0.372257,-0.232466)%
(0.402897,-0.241935)(0.433477,-0.250115)(0.463924,-0.257002)(0.494166,-0.262596)(0.524129,-0.266899)%
(0.553741,-0.26992)(0.582931,-0.271668)(0.611627,-0.272158)(0.639761,-0.271409)(0.667263,-0.269441)%
(0.694067,-0.266279)(0.720106,-0.261953)(0.745318,-0.256493)(0.769639,-0.249935)(0.79301,-0.242317)%
(0.815373,-0.23368)(0.836674,-0.224067)(0.85686,-0.213526)(0.87588,-0.202107)(0.893687,-0.18986)%
(0.910238,-0.17684)(0.925491,-0.163105)(0.939408,-0.148711)(0.951955,-0.13372)(0.9631,-0.118193)%
(0.972817,-0.102195)(0.981079,-0.0857896)(0.987868,-0.0690433)(0.993166,-0.0520232)(0.996959,-0.0347969)%
(0.999239,-0.0174329)(1,0)%
(0.999239,0.0174329)(0.996959,0.0347969)(0.993166,0.0520232)(0.987868,0.0690433)(0.981079,0.0857896)%
(0.972817,0.102195)(0.9631,0.118193)(0.951955,0.13372)(0.939408,0.148711)(0.925491,0.163105)(0.910238,0.17684)%
(0.893687,0.18986)(0.87588,0.202107)(0.85686,0.213526)(0.836674,0.224067)(0.815373,0.23368)(0.79301,0.242317)%
(0.769639,0.249935)(0.745318,0.256493)(0.720106,0.261953)(0.694067,0.266279)(0.667263,0.269441)(0.639761,0.271409)%
(0.611627,0.272158)(0.582931,0.271668)(0.553741,0.26992)(0.524129,0.266899)(0.494166,0.262596)(0.463924,0.257002)%
(0.433477,0.250115)(0.402897,0.241935)(0.372257,0.232466)(0.34163,0.221715)(0.311088,0.209696)(0.280704,0.196422)%
(0.250548,0.181912)(0.220691,0.166189)(0.191201,0.149279)(0.162148,0.131212)(0.133596,0.11202)(0.105611,0.0917388)%
(0.0782552,0.0704087)(0.0515901,0.0480719)(0.025674,0.0247738)(0.000563312,0.000562864)
}}
\end{pspicture}
%}
\end{minipage}
\sk{-9}\noindent
on $\ga$, then $C=0$ by symmetry (since $\ga$ and $T$ are invariant 
under the reflection through the horizontal axis).
If the angle at $C$ is obtuse, this is impossible since $\ga$ lies in in the polar domain $|\theta|\leq\pi/4$.
Note that $\ga$ is of class $\calc^\infty$ except at $0$, where the two tangents form a right angle.   

However, a Jordan $\calc^0$-curve contains any non-flat triangle up to similarity. 
This result was established in \cite{Me} with an elementary proof
(see also \cite[Section~11, Theorem~1.3]{KW}).

Compared to the first version of this paper, a new section (\secref{S.gener}) has been added,
containing a generalization of the main theorem for $n$-simplexes on Jordan spheres in $\bbr^n$
($n\neq 4$). The proof just requires slight adaptations.
I am grateful to Michelle Bucher-Karlsson for making me aware of such an extension.

%\fi

The proof of the main theorem is given in \secref{S.pfmainth} while the previous ones 
are devoted to preliminary material. A version of the proof was written by V\'eronique Gonoyan 
in her master thesis (University of Geneva, 2003). Discussions with Anton Alekseev were useful.

\section{The space of triangles}\label{S.sptri}

Identifying the plane with $\bbc$, the \dfn{space of triangles} (not reduced to a single point) $\tri^0$ is the smooth manifold
$$
\tri^0=\{(z_0,z_1,z_2))\in\bbc^3-\Delta \} \, ,
$$ 
where $\Delta=\{(z,z,z)\}$ is the diagonal subset of $\bbc^3$. 
Let $\calg_1\approx \bbc$ be the group of translations of $\bbc$. The diagonal $\calg_1$-action on  $\bbc^3$ preserves
$\Delta$ and is free and proper. Hence, $\tri^1=\tri^0/\calg_1$ inherits a structure of a smooth manifold and 
the correspondence $(z_0,z_1,z_2)\mapsto (z_1-z_0,z_2-z_0)$ induces a diffeomorphism
$$
\tri^1 \approx \bbc^2-\{(0,0)\} \, .
$$
Let $\calg_2$ be the group of homothety of $\bbc$ with a positive ratio (isomorphic to the multiplicative group $\bbr_{>0}$).
Again, 
$$
\tri = \tri^1/\calg_2
$$
is a smooth manifold and one has a diffeomorphism
$$
\tri \approx \big(\bbc^2-\{(0,0)\}\big)\Big/\bbr_{>0} \approx S^3 
$$
from $\tri$ to the standard sphere $S^3$. 

Finally, the group $\calg_3\approx S^1$ of rotations of $\bbc$ acts on $\tri^1$ and $\tri$
and one has diffeomorphisms
$$
\tri\big/\calg_3 \approx S^3/S^1 \approx \bbc P^1 \approx \hat\bbc \, ,
$$ 
where $\hat\bbc=\bbc\cup\{\infty\}$. The following diagram
$$
\dia{
\tri^1 \ar[d]^\approx  \ar@{>>}[r]  &  \tri \ar@{>>}[r] & \tri\big/\calg_3 \ar[d]^\approx \\
%%%ROW2
\bbc^2-\{(0,0)\} \ar@{>>}[rr]^{\beta}  && \hat\bbc
}
$$
is commutative, where $\beta$ is the classical Hopf map
$$
\beta(z_1,z_2) =
\left\{ 
\begin{array}{lll}
{\displaystyle \frac{z_1}{z_2}} & \hbox{if $z_2\neq 0$} \\[2mm]
\infty &  \hbox{otherwise.}
\end{array}\right.
$$

Let $\ftri^0$ be the subspace of $\tri^0$ formed by those triangles which are \dfn{flat}, namely contained in a line.
Denote its images in $\tri^1$ (respectively: in $\tri$) by $\ftri^1$ and  (respectively: $\ftri$). 
The space $\ftri^1$ in $\tri^1\approx \bbc^2-\{(0,0)\}$ is formed by the couples $(z_1,z_2)$ 
of complex numbers which are $\bbr$-linearly dependent. By the above definition of $\beta$, 
the image of $\ftri$ in $\tri/\calg_3\approx\hat\bbc$ is equal to $\hat\bbr=\bbr\cup\{\infty\}$. 
As $\beta$ is a circle bundle, the surface $\calt=\beta^\mun(\hat\bbr)$ is a circle bundle over  $\hat\bbr\approx S^1$.
Since $\tau$ separates $S^3$, it is orientable and thus diffeomorphic to a $2$-torus. This proves the following proposition.

\begin{Proposition}\label{P.tritrifl}
There is a diffeomorphism of manifold pairs
$$
(\tri,\ftri) \approx (S^3,\calt)
$$ 
where $\calt$ is a $2$-torus in $S^3$, separating $S^3$ into two components. \mancqfd
\end{Proposition}

\begin{Remark}\rm The reader may easily check the following facts about the images in $\tri/\calg_2\approx\hat\bbc$
of the following subsets of $\tri^0$.
\begin{itemize}
\item the flats triangles $(z_0,z_1,z_2)$ with two identical vertices have image equal to $1$ (if $z_1=z_2$), 
$0$ (if $z_0=z_1$) and $\infty$ (if $z_0=z_2$).
\item The equilateral triangles have image $e^{\pm i\pi/3}$.
\item The isosceles triangles have image the circles of equation $|z|=1$, $|z-1|=1$ and the vertical line through $1/2$
(union $\{\infty\}$).
\item The rectangles triangles have image the circle $|z-1/2|=1/2$ and the two vertical lines (union $\{\infty\}$)
through $0$ and $1$.
\end{itemize}
\end{Remark}

%%%%%%%%%%%%%%%%%%%%%%%%%%%%%%PICTURE

\sk{45}

\hskip 40mm
\begin{minipage}{6cm}
%\begin{center}
\hskip 10mm
\setlength{\unitlength}{.3mm}
\scalebox{1.0}{
\begin{pspicture}(0,-2.1)(0,0)
%\grille

\psline[linewidth=2pt](-4,0)(5,0)
\psline[linestyle=dotted,linewidth=1pt](0,-3.5)(0,3.5)
\psline[linestyle=dotted,linewidth=1pt](2,-3.5)(2,3.5)
\put(1,0){\pscircle[linestyle=dotted]{1}}
\put(0,0){\pscircle[linestyle=dashed]{2}}
\put(2,0){\pscircle[linestyle=dashed]{2}}
\psline[linestyle=dashed,linewidth=1pt](1,-4)(1,3.9)
\put(1,1.732){\pscircle*[linestyle=dashed]{0.12}} 
\put(1,-1.732){\pscircle*[linestyle=dashed]{0.12}}

\pspolygon[linewidth=1pt](1.1,2.1)(1.5,2.1)(1.3,2.38)
\pspolygon[linewidth=1pt](1.1,-2.1)(1.5,-2.1)(1.3,-2.38)

\put(0,0){\pscircle*[linecolor=black]{0.1}} 
\put(2,0){\pscircle*[linecolor=black]{0.1}}

\uput{0mm}[90](-0.2,0.2){$0$}
\uput{0mm}[90](1.8,0.2){$1$}
\uput{0mm}[90](-1.73,0.1){$-1$}

\psline[linewidth=1pt](-0.3,-0.2)(-0.7,-0.2)
\put(-0.3,-0.2){\pscircle*[linecolor=black]{0.06}} 
\put(-0.4,-0.55){{\small $C$}}
\put(-0.7,-0.22){\pscircle*[linecolor=black]{0.06}} 
\put(-1.1,-0.55){{\small $AB$}}
\put(-0.7,-0.18){\pscircle*[linecolor=black]{0.06}} 

\psline[linewidth=1pt](2.3,-0.2)(2.7,-0.2)
\put(2.3,-0.2){\pscircle*[linecolor=black]{0.06}} 
\put(2.7,-0.22){\pscircle*[linecolor=black]{0.06}} 
\put(2.7,-0.18){\pscircle*[linecolor=black]{0.06}} 
\put(2.15,-0.55){{\small $A$}}
\put(2.6,-0.55){{\small $BC$}}

\psline[linewidth=1pt](-2.2,-0.2)(-3,-0.2)
\put(-2.2,-0.2){\pscircle*[linecolor=black]{0.06}} 
\put(-2.6,-0.2){\pscircle*[linecolor=black]{0.06}} 
\put(-3,-0.2){\pscircle*[linecolor=black]{0.06}} 
\put(-2.74,-0.55){{\small $A$}}
\put(-2.32,-0.55){{\small $B$}}
\put(-3.15,-0.55){{\small $C$}}

\put(4.3,-0.3){$\infty =$}
\psline[linewidth=1pt](5.2,-0.2)(5.6,-0.2)
\put(5.2,-0.22){\pscircle*[linecolor=black]{0.06}} 
\put(5.2,-0.18){\pscircle*[linecolor=black]{0.06}}
\put(5.6,-0.2){\pscircle*[linecolor=black]{0.06}} 
\put(4.9,-0.55){{\small $AC$}}
\put(5.55,-0.55){{\small $B$}}

\psline[linestyle=dashed,linewidth=1pt](-4.5,3.3)(-3.8,3.3)
\put(-3.5,3.2){{isosceles triangles}}
\psline[linestyle=dotted,linewidth=1pt](-4.5,2.6)(-3.8,2.6)
\put(-3.5,2.5){{rectangle triangles}}

\put(-3.5,0.16){${\rm Tri}_{fl}$}

\end{pspicture}
}
\end{minipage}

\sk{30}
By a Moebius transformation $h$ of $\bar\bbr^3$, one can move $\bar\bbc$ onto the unit sphere~$S^2$, so that
the equilateral triangles are the poles and $\ftri$ is the equator. The locus of isosceles triangles is
then formed by three meridians, dividing $S^2$ into six equal sectors (since $h$ preserves the angles). 
Such pictures are used in statistical shape analysis (see e.g. \cite[p.~37]{DM}).

%%%%%%%%%%%%%%%%%%%%%%%%%%%%%%%%%%%%%%%%%%%%%%%%%%
\section{The main map}\label{S.mainmap}

\subsection{Definitions}
Let $c\:S^1\to\bbc$ be a $\calc^1$-embedding (injective $\calc^1$-immersion), parametrizing a Jordan curve $\ga$
and let $\tilde c\:\bbr\to\bbc$ be defined by $\tilde c(t)=c(e^{2i\pi t})$. That $c$ is an immersion is
equivalent to $\dtc(t)\neq 0$ for all $t\in\bbr$.
   
Let $V=(S^1)^3-\Delta$ and define the map $F^0\:V\to \tri^0$ by 
$$
F^0(s_0,s_1,s_2) = (c(s_0),c(s_1),c(s_2)) \, .
$$
Taking the images in $\tri^1$ and $\tri$ provides maps
$$
F^1\:V\to \tri^1 \quad , \quad F^1(s_0,s_1,s_2)=(c(s_1)-c(s_0),c(s_2)-c(s_0)) 
$$
and
\beq{E.dfc}
F\:V\to \tri \, .
\eeq
One can pre-compose the above maps with $\exp\:\bbr^3\to (S^1)^3$, the universal covering defined by
$\exp(t_0,t_1,t_2)=(e^{2i\pi t_0},e^{2i\pi t_1},e^{2i\pi t_2})$. This provides maps $F^1_e\:W\to \tri^1$
and $F_e\:W\to \tri$, where
$W=\bbr^3-\tilde\Delta$ with
$$
\tilde\Delta=\exp^\mun(\Delta)=\{(t+p,t+q,t+r)\in\bbr^3 \mid  t\in\bbr \hbox{ and } (p,q,r)\in\bbz^3) \} \, .
$$
All these maps are of class $\calc^1$ and the maps $F^0_e$, $F^1_e$ and $F_e$ are invariant under the 
$\bbz^3$-action on $W$ by translation. It will be useful to know what are the critical points of $F$.
Recall that a point $x\in M$ is \dfn{critical} for a $\calc^1$-map $f\:M\to N$ between manifolds whenever the 
tangent map $T_xf\:T_xM\to T_{f(x)}N$ is not surjective.

\begin{Proposition}\label{P.crit} 
A point $(s_0,s_1,s_2)\in (S^1)^3-\Delta$ is a critical point for the map $F$ if and only if 
the tangents to $\ga$ at the points $c(t_0)$, $c(t_1)$ and $c(t_2)$ are parallel or concurrent.
\end{Proposition}

\begin{proof}
It is equivalent to prove the statement for the map $F_e$, since $\exp$ is a covering. 
Since $\dim W=\dim\tri$, a point $(t_0,t_1,t_2)\in W$ is critical for $F_e$ if and only if 
$T_{(t_0,t_1,t_2)}F_e$ is not injective.
Recall that, if $U$ is an open subset of a real vector space $X$, the tangent space 
$T_uU$ at each point $u\in U$ is canonically identified with $X$. Under such identification,
one has
$$%\beq{P.crit-eq15}
T_{(t_0,t_1,t_2)}F^1_e(\lambda_0,\lambda_1,\lambda_2) = 
(\lambda_1\,\dtc(t_1)-\lambda_0\,\dtc(t_0),\lambda_2\,\dtc(t_2)-\lambda_0\,\dtc(t_0)) \, .
$$%\eeq
The point $(t_0,t_1,t_2)\in W$ is critical for $F$ if and only if 
$$
T_{(t_0,t_1,t_2)}F^1_e(\lambda_0,\lambda_1,\lambda_2) \in \ker T_{F_e(t_0,t_1,t_2)} \pi
$$
for all $(\lambda_0,\lambda_1,\lambda_2)\in\bbr^3$, where $\pi\:\tri^1\to\tri$ is the quotient map.
But
$$
\ker T_{(z_1,z_2)} \pi = \bbr\cdot (z_1,z_2) \, .
$$
Indeed, $\bbr_{>0}$-action on $\tri^1$ corresponds to the ``dilatation flow'' $\Phi_t(z_1,z_2)=t(z_1,z_2)$
and $\frac{d}{dt}\Phi_t(z_1,z_2)_{|t=0}=(z_1,z_2)$. Therefore, $(t_0,t_1,t_2)\in W$ is critical for $F$ if and only if 
there exists $(\lambda_0,\lambda_1,\lambda_2)\in\bbr^3$ and $«\lambda\in\bbr$ such that
\beq{P.crit-eq20}
\left\{
\begin{array}{rcl}
\lambda_1\,\dtc(t_1)-\lambda_0\,\dtc(t_0) &=& \lambda(\tilde c(t_1)-\tilde c(t_0)) \\[2mm]
\lambda_2\,\dtc(t_2)-\lambda_0\,\dtc(t_0) &=& \lambda(\tilde c(t_2)-\tilde c(t_0))  \, .
\end{array}\right.
\eeq

If $\lambda=0$, the above system is equivalent to 
$$
\lambda_1\,\dtc(t_1)=\lambda_0\,\dtc(t_0)=\lambda_2\,\dtc(t_2)
$$
which is equivalent to the parallelism of the tangents to $\ga$ at the points $c(t_0)$, $c(t_1)$ and $c(t_2)$.
If $\lambda\neq 0$, then replacing $\lambda_j$ by $\pm\lambda_j/\lambda$ in~\eqref{P.crit-eq20} makes the system
equivalent to
$$
\tilde c(t_1) + \mu_1\,\dtc(t_1) = \tilde c(t_0) + \mu_0\,\dtc(t_0) = \tilde c(t_2) + \mu_2\,\dtc(t_2) 
$$
for some $(\mu_0,\mu_1,\mu_2)\in\bbr^3$.
This is equivalent to the concurrency of the three tangents.
%to $\ga$ at the points $c(t_0)$, $c(t_1)$ and $c(t_2)$.
\end{proof}

\subsection{Compactifications}\label{S.compa}
We now define boundary compactifications $W\subset \hat W$ and $V\subset\hat V$ and extend the maps $F_e$ and $F$ 
to the manifolds with boundary $\hat W$ and $\hat V$.
Let 
$D^2_{\scr 1\!/\!2}=\{(u,v)\in\bbr^2\mid u^2+v^2<1/4\}$ be the disk of radius $1/2$, with boundary $S^1_{\scr 1\!/\!2}$. 
Let us parameterize an open tubular neighborhood of $\tilde\Delta$ in $\bbr^3$ by the embedding
$h\:\tilde\Delta\times S^1_{\scr 1\!/\!2}\times [0,1)\to \bbr^3$ defined by  
$$
h(t+p,t+q,t+r,(u,v),\lambda) = (t+p,t+q+\lambda u,t+r+\lambda v)  \, .
$$
Define
$$
\hat W = \Big(\big\{\tilde\Delta\times S^1_{\scr 1\!/\!2}\times [0,1)\big\} \,\,\dot\cup \,\, W\Big) \, \Big/  \sim  \, ,
$$
where ``$\sim$'' is the smallest equivalence relation such that
\beq{E.dsim}
((t+p,t+q,t+r),(u,v),\lambda) \sim h(t+p,t+q,t+r,(u,v),\lambda)  \ \hbox{ when } \ \lambda\neq 0  \, .
\eeq

As the map $h$ is of class $\calc^\infty$, we check that $\hat W$ is a $\calc^\infty$-manifold with boundary 
$\partial\hat W= \tilde\Delta\times S^1_{\scr 1\!/\!2}\times \{0\}$. 
The inclusion $W\to \hat W$ is a diffeomorphism onto $\hat W-\partial\hat W$.
The inclusion $\tilde\Delta\times S^1_{\scr 1\!/\!2}\times [0,1)\to \hat W$ 
is an embedding onto an open collar of $\partial\hat W$.

The above construction is invariant under the action of $\bbz^3$ on $\bbr^3$ by translations, so $\bbz^3$ acts freely and
properly on $\hat W$. The quotient $\hat V = \hat W/\bbz^3$ is a compact $\calc^\infty$-manifold, 
with boundary $\partial\hat V\approx S^1\times S^1$ and with a diffeomorphism from $V$ onto $\hat V-\partial\hat V$.

\begin{Lemma}\label{L.hatF}
Let $c\:S^1\to\bbc$ be a $\calc^1$-embedding. Then, the map $F$ of~\eqref{E.dfc} extends to a continuous map
of pairs
$$
\hat F \: (\hat V,\partial\hat V) \to (\tri,\ftri) \, .
$$
\end{Lemma}

\begin{proof} 
We define a $\bbz^3$-invariant continuous extension $\hat F_e\:(\hat W,\partial\hat W)\to (\tri,\ftri)$. 
This provides $\hat F$ by passing to the quotient.  

The map $\hat F_e$ restricts to $F_e$ on $W$. Therefore, on $\tilde\Delta\times S^1_{\scr 1\!/\!2}\times (0,1)$, 
it must be equal to $F_e\pcirc h$. Let $\tilde c\:\bbr\to\bbc$ defined by $\tilde c(t)=c(e^{2i\pi t})$. 
Note that $\tilde c(t+m)=c(t)$ for $m\in\bbz$.
Hence, for $Z= ((t+p,t+q,t+r),(u,v),\lambda)$ and $\lambda\neq 0$, one has
\beq{L.hatF-eq1}
\begin{array}{rcll}
F^1_e\pcirc h(Z)  
 &=&
(\tilde c(t+\lambda u)- \tilde c(t),\tilde c(t+\lambda v)- \tilde c(t))  \\[2mm] &=&
(\lambda u\,\dtc(t) + \pcirc(\lambda u) ,\lambda v\,\dtc(t) + \pcirc(\lambda v)) & \hbox{\small since $c$ is $\calc^1$}
\\[2mm] &=&
(\lambda( u\,\dtc(t) + \frac{\pcirc(\lambda u)}{\lambda}) ,\lambda( v\,\dtc(t) + \frac{\pcirc(\lambda v)}{\lambda})) 
\end{array}
\eeq
The last expression has same image in $\tri$ as 
$(u\,\dtc(t) + \frac{\pcirc(\lambda u)}{\lambda} ,v\,\dtc(t) + \frac{\pcirc(\lambda v)}{\lambda})$ and, in $\tri^1$, one has
the following convergence
\beq{L.hatF-eq12}
\dia{
(u\,\dtc(t) + \frac{\pcirc(\lambda u)}{\lambda} ,v\,\dtc(t) + \frac{\pcirc(\lambda v)}{\lambda}) \kern 2mm
\ar[r]_(.56){(\lambda\to 0)} & 
\kern 2mm (u\,\dtc(t) ,v\,\dtc(t) )  \in \ftri^1 \, .
}
\eeq
We are thus driven to define
$$
\hat F_e\pcirc h(Z) = 
\left\{
\begin{array}{ll}
F_e\pcirc h(Z) & \hbox{if  $Z\in W$} \\[1mm]
 [(u\,\dtc(t) ,v\,\dtc(t) )] & \hbox{if  $\lambda=0$ }\, ,
\end{array}\right.
$$
where $[\,]$ denotes the class in $\tri$. To check the continuity of $\hat F_e$, we have to take 
a converging sequence $Z_n\to Z_\infty$ in $\hat W$ and see that $\hat F_e\pcirc h(Z_n)\to \hat F_e\pcirc h(Z_\infty)$.
Set
$$
Z_n= ((t_n+p_n,t_n+q_n,t_n+r_n),(u_n,v_n),\lambda_n)
$$ and $$
Z_\infty= ((t_\infty+p_\infty,t_\infty+q_\infty,t_\infty+r_\infty),(u_\infty,v_\infty),\lambda_\infty) \, .
$$
Only the case $\lambda_\infty=0$ requires a proof.
As $p_n$, $q_n$ and $r_n$ are integers which play no role, one may assume that $(p_n,q_n,r_n)=(0,0,0)$.
Since $\hat F_e|\partial \hat W$ is continuous, one may assume, restricting to a subsequence if necessary, 
that $\lambda_n\neq 0$ if $n<\infty$. Let us decompose $\tilde c$ in its real and imaginary part:
$\tilde c(t)=\tilde c_{\rm re}(t) + i\,\tilde c_{\rm im}(t)$. By the mean value theorem, one has
$$
\dtc_{\rm re}(t_n+\lambda_nu_n)-\dtc_{\rm re}(t_n) = \lambda_nu_n\,\dtc_{\rm re}(t_n+\mu_n)
$$
with $|\mu_n|\leq |\lambda_nu_n|$. We can write the analogous equation for 
$\dtc_{\rm re}(t_n+\lambda_nv_n)-\dtc_{\rm re}(t_n)$ and for the imaginary parts, and use them
in the computations like in~(2.6--7). This proves that $\hat F_e\pcirc h(Z_n)\to \hat F_e\pcirc h(Z_\infty)$.
\end{proof}

%%%%%%%%%%%%%%%%%%%%%%%%%%
\section{Bidegree}

In this section $H_*(\,)$ denotes the singular homology with $\bbz_2$ as coefficients. 
The various manifolds are topological manifolds and may be non-orientable.
The following lemma will be useful.

\begin{Lemma}\label{L.Hnvanish}
Let $X$ be a compact connected topological $n$-manifold with (possibly empty) boundary $\partial X$. 
Let $K$ be a non-empty discrete set in $X-\partial X$. Then $H_n(X-K,\partial X)=0$.
\end{Lemma}

\begin{proof}
We may suppose that $n>1$, otherwise the easy proof is left to the reader.
Since $X$ is compact, $K$ is finite. Let $\cald\subset X-\partial X $ be a disjoint family of compact disks 
forming a tubular neighborhood of $K$. Then $Y=X-{\rm int\, }\cald$ is a compact connected manifold with
$\partial Y$ being the disjoint union of $\partial X$ and $\partial\cald$. 
Since the inclusion $Y\subset X-K$ is a homotopy equivalence, one has $H_n(X-K,\partial X)\approx H_n(Y,\partial X)$.
By Poincar\'e duality (see e.g. \cite[Theorem~]{Hatcher} or \cite[Corollary~29.11]{Ha}), $H_n(Y,\partial X)\approx H^0(Y,\partial\cald)$.
As $n>1$, $Y$ is connected. Therefore, since $K$ is not empty, so is $\partial\cald$ and thus $H^0(Y,\partial\cald)=0$.
\end{proof}

Let $M$ be a closed connected manifold of dimension $m$, containing a closed submanifold $N$ of codimension one which separates $M$ into 
two connected manifolds $M_\pm$, with common boundary $N$. Let $V$ be a tubular neighborhood of $N$ in $M$.
By homotopy and excision, one has the isomorphisms
 $$
H_n(M,N) \approx H_n(M,V) \approx H_n(M-{\rm int} V,\partial V) \, . %\approx H_n(M_-,N)  \oplus H_n(M_+,N)  \approx \bbz_2\oplus\bbz_2 \, .
$$
As $N$ separates, $V$ is of the form $N\times [-1,1]$ with $N\times\{\pm 1\} \subset M_\pm$ and
one has a homotopy equivalence of pairs
$$
\begin{array}{rcl}
(M-{\rm int} V,\partial V) &=& (M_--{\rm int} V,N\times\{-1\})\, \dot\cup\, (M_+-{\rm int} V,N\times\{1\}) 
\\[2mm] &\simeq &
(M_-,N) \,\dot\cup \, (M_+,N)
\end{array}
$$
Hence,
$$
H_n(M,N) \approx H_n(M_-,N)  \oplus H_n(M_+,N)  \approx \bbz_2\oplus\bbz_2 \, .
$$

Let $X$ be a compact connected manifold of dimension $n$ and let  $f\:(X,\partial X)\to (M,N)$ be a continuous map of pairs.
Consider the following commutative diagram
$$
\dia{
H_n(X,\partial X) \ar[r]^{f_*} \ar[d]^{\approx} &  H_n(M,N) \ar[d]^{\approx} \\
%%%ROW2
\bbz_2 \ar[r]^{\hat f_*} & \bbz_2\oplus \bbz_2
} .
$$
The couple 
$$
\bdg(f)=\hat f_*(1)\in \bbz_2\oplus \bbz_2
$$ 
is called the \dfn{bidegree} of $f$. 
 
\begin{Proposition}\label{P.bidegSurj}
If $\bdg(f)=(1,1)$, then $f$ is surjective.
\end{Proposition}

\begin{proof}
Suppose that $f$ is not surjective. As $X$ is compact, the set of points $u\in M$ with empty pre-image is open. 
Hence, there is such a point in $u\in M-N$, say $u\in M_+$. But $H_n(M_+-\{u\},N)=0$ by \lemref{L.Hnvanish}. 
Therefore, $\bdg(f)\neq (1,1)$.
\end{proof}

The bidegree may be computed locally.
A point $u\in M-N$ is a \dfn{topological regular value}
for $f$
if there is a neighborhood $U$ of $u$ in $M-N$ such that 
$U$ is ``\dfn{evenly covered}'' by $f$. By this, we mean that
$f^\mun(U)$ is a disjoint union of subspaces $U_j$, indexed by a set $\calj$,
such that, for each $j\in\calj$, the restriction of $f$ to $U_j$ 
is a homeomorphism from $U_j$ to $U$.
In particular, $f^\mun(u)$ is a discrete closed subset of $X$
indexed by $\calj$, so $\calj$ is finite since $X$ is compact.
For instance, a point $u$ which is not in the range of $f$ is a topological regular
value of $f$ (with $\calj$ empty). 
For a topological regular value $u$ of $f$, we define the \dfn{local degree 
$d(f,u)\in\bbz_2$ of $f$ at $u$} by 
$$
d(f,u)=\sharp f^\mun(u) \quad {\rm mod\,}2  \, .
$$

\begin{Proposition}\label{P.degree-sing}
Let $f\:(X,\partial X)\to (M,N)$ as above.
For any topological regular values $u_\pm\in M\pm$ of $f$, one has
$$
\bdg(f)=(d(f,u_-),d(f,u_+)).
$$
\end{Proposition}

\begin{proof}
We prove that the $H_n(M_+,N)$-component of $\bdg(f)$ is equal to $d(f,u_+)$. The argument for the other
component is the same. The inclusions of pairs $i_\pm\:(M_\pm,N)\to (M,N)$ and $j_\pm\:(M,N)\to (M,M_\pm)$
induce homomorphisms
$$
\dia{
H_n(M_-,N)\oplus H_n(M_+,N) \ar[r]^(0.6){H_*i_-+H_*i_+}  \ar[dr]_{\approx}
& H_n(M,N)  \ar[d]^{(H_*j_+,H_*j_-)} \\
%%%ROW2
& H_n(M,M_+)\oplus H_n(M,M_-)
}
$$
whose composition is an isomorphism by excision. The three groups above being isomorphic to $\bbz_2\oplus\bbz_2$,
all the arrows are  isomorphisms. Hence, the $H_n(M_+,N)$-component of $\bdg(f)$ is equal
to $H_*j_-\pcirc \hat f(1)$. One may suppose that $u_+$ has a neighborhood $D$ evenly 
covered by $f$ which is homeomorphic to a compact $n$-disk. Let $\calu=f^\mun(\{u_+\})$ and $\cald=f^\mun(D)$.
One has the commutative diagram
$$
\dia{
H_n(X,\partial X) \ar[r]^{H_*f} \ar@{>->}[d] &
H_n(M,N) \ar[r]^{H_*j_-} &
H_n(M,M_-) \ar[d]^{\approx}  &
\\
%%%ROW2
H_n(X,X-\calu) \ar[rr]^{H_*f} &&
H_n(M,M-\{u_+\})  \\
%%%ROW3
H_n(\cald,\partial\cald) \ar[u]_{\approx}  \ar[rr]^{H_*f_\cald} &&
H_n(D,\partial D) \ar[u]_{\approx} 
}
$$
where the bottom vertical arrows are isomorphisms by excision. 
The top left vertical arrows is injective, since, in the 
homology exact sequence of the triple $(X,X-\calu,\partial X)$
$$
\dia{
H_n(X-\calu,\partial X) \ar[r] &  H_n(X,\partial X)  \ar[r]^(.45){} &   H_n(X,X-\calu) \, ,
}
$$ 
the left term vanishes by \lemref{L.Hnvanish} (one may suppose that $\calu\neq\emptyset$ since, otherwise
the proof is trivial). 
The top right vertical arrows is injective by the same argument, so it is an isomorphism 
since its range is isomorphic to $\bbz_2$. Writing $\cald$ as a disjoint union of compact disks
$\cald_1,\dots,\cald_k$, one has a commutative diagram
$$
\dia{
H_n(\cald,\partial\cald) \ar[d]^{H_*f_\cald} \ar@{<-}[r]^(0.45){\approx} &
\bigoplus_{j=1}^k H_n(\cald_j,\partial\cald_j) \ar[d]^{\sum H_*f_{D_j}} \ar[r]^(0.6){\approx} &
\bigoplus_{j=1}^k \bbz_2 \ar[d]^{+}
\\
%%%ROW2
H_n(D,\partial D) \ar@{<->}[r]^{=} & H_n(D,\partial D) \ar[r]^(0.55){\approx} & \bbz_2
}
$$
which proves that the $H_n(M_+,N)$-component of $\bdg(f)$is equal to $d(f,u_+)$.
\end{proof}

\begin{Remark}\label{R.hty} \rm
Let $f,g\:(X,\partial X)\to (M,N)$ be two maps as above which are homotopic (as maps of pairs).
As $f_*=g_*$, one has $\bdg(f)=\bdg(g)$. \mancqfd 
\end{Remark}

\section{Proof of the main theorem}\label{S.pfmainth}

Let $c\:S^1\to\bbc$ be a $\calc^1$-embedding parametrizing $\ga$.  
Denote by $F_c\,\: V\to \tri$ the map $F$ of~\eqref{E.dfc} for the parametrization $c$ and
by $\hat F_c \,\: (\hat V,\partial\hat V) \to (\tri,\ftri)$ its continuous extension given by \lemref{L.hatF}.
It suffices to prove that $\bdg(\hat F_c)=(1,1)$, which, using \proref{P.bidegSurj}, implies that $\hat F_c$ is surjective.
Indeed, as $\hat F_c(\partial \hat V)\subset \ftri$, the class of a non-flat triangle will be in $F_c(V)$, which
proves the main theorem.

Consider particular case where $\ga=S^1$, i.e. $c$ is a $\calc^1$-parametrisation of the unit circle. 
By \proref{P.crit}, any non-flat triangle $T$ on $\ga$ is a smooth regular value of $F_c$ and $F_c^\mun(T)$ contains a single point.
By the inverse function theorem, $T$ is a topological regular value of $\hat F_c$. 
Using \proref{P.degree-sing} for $T$ and its conjugate $\bar T$, this proves that $\bdg(\hat F_c)=(1,1)$.

For a general simple closed curve $\ga$ of class $\calc^1$, we shall prove that there is an isotopy 
of $\calc^1$-embeddings $c_t\:S^1\to\bbc$ satisfying $c_0=c$ and $c_1(S^1)=S^1$. The map
$\hat F_{c_t}\:(\hat V,\partial\hat V) \to (\tri,\ftri)$ is then a homotopy of maps of pairs between 
$\hat F_c$ and $\hat F_{c_1}$
(the continuity of $\hat F_{c_t}$ may be checked with the arguments of the end of the proof of \lemref{L.hatF}). 
Using \remref{R.hty} and that $c_1(S^1)=S^1$, this will prove that $\bdg(\hat F_c) =\bdg(\hat F_{c_1}) = (1,1)$. 

The construction of the isotopy $c_t$ proceeds as follows.
\begin{itemize}
 \item[(a)] If $\gamma\:S^1\to\bbc$ is a $\calc^1$-map which is close enough to $c$ in the $\calc^1$-metric,
then $\gamma$ is an injective immersions and $tc+(1-t)\gamma$ produces a $\bbc^1$-isotopy between $c$ and $\gamma$.
As such a $\gamma$ may be chosen of class $\calc^\infty$ \cite[p.~49]{Hirsch},
 we may thus suppose that $c$ is of class $\calc^\infty$.

\item[(b)] By the Schoenflies theorem \cite[Theorem 5.4]{KC}, the embedding $c$ extends to a 
$\calc^\infty$-embedding $\bar c\:D\to\bbc$, were $D$ is the unit disk. 
Such an embedding is isotopic to a diffeomorphism of $D$. 
Indeed, composing with a translation (isotopic to the identity), one may suppose that $\bar c(0)=0$.
Then the map
$$
\bar c_t(z) = 
\left\{\begin{array}{ll}
\frac{1}{t}\bar c(tz) & \hbox{if $t\neq 0$} \\[2mm]
D_0\bar c(z)  & \hbox{if $t=0$.}
\end{array}\right.
$$ 
\end{itemize}
is a $\calc^\infty$-isotopy between $\bar c$ and a $\bbr$-linear embedding. But $GL(2,\bbr)$
has two connected components, both containing isometries.

\begin{Remark}\label{R.schoe2}\rm
The proof of Schoenflies theorem given in \cite[Theorem 5.4]{KC} is not detailed for the smooth case. 
At least, the topological Schoenflies theorem implies that $c(S^1)$ bounds a disk.
The embedding $\bar c$ may then be obtained by the Riemann mapping theorem with its extension to the boundary
(see e.g. \cite{BellKrantz}).
\end{Remark}

% \sk{2}
% Remark: the proof of Schoenflies theorem given in \cite[Theorem 5.4]{KC} is not detailed for the smooth case. 
% At least, the topological Schoenflies theorem implies that $c(S^1)$ bounds a disk.
% The embedding $\bar c$ may then be obtained by the Riemann mapping theorem with its extension to the boundary
% (see e.g. \cite{BellKrantz}).
%(see e.g.~\cite[p.~296]{Gamelin})

\section{Generalization to higher dimensions}\label{S.gener}

%This section was written following remarks by Michelle Bucher-Karlsson. \cb

A \dfn{Jordan $\calc^1$-sphere $\ga$} in $\bbr^n$ is a 
$\calc^1$-submanifold of $\bbr^n$ diffeomorphic to the standard sphere $S^{n-1}$. 
An $n$-simplex of $\bbr^n$ is \dfn{flat} if it is contained in an affine subspace of dimension $n-1$.
The main theorem admits the following generalization, as suggested by Michelle Bucher-Karlsson.

\begin{Theorem}\label{mainThGen}
%Consider in $\bbr^n$ a Jordan $\calc^1$-sphere $\ga$ and a non-flat $n$-simplex~$T$.
Let $\ga$ be a Jordan $\calc^1$-sphere in $\bbr^n$ and let $T$ be a non-flat $n$-simplex in $\bbr^n$.
Suppose that $n\neq 4$.
Then, by a translation and a homothety with positive ratio, $T$ may be transformed into a simplex whose
vertices lie on~$\ga$. 
\end{Theorem}

The proof of \thref{mainThGen} requires the following smooth Schoenflies theorem
(for $n=4$, see \remref{R.schn=4}).

\begin{Proposition}[Smooth Schoenflies theorem]\label{P.schoe}
Let $\ga$ be a Jordan $\calc^\infty$-sphere in $\bbr^n$ with $n\neq 4$. Then $\ga$ bounds a smooth disk in $\bbr^n$.
\end{Proposition}

\begin{proof}
The classical case $n=2$ was recalled in \secref{S.pfmainth} (see (b) and \remref{R.schoe2}).
Let $c\:S^{n-1}\to\bbr^n$ be a smooth ($\calc^\infty$) embedding whose image is $\ga$.
By the tubular neighborhood theorem, $c$ extends to an embedding of $S^{n-1}\times[-1,1]$ into $\bbr^{n+1}$.
The generalized Schoenflies theorem then holds \cite[Theorem~5]{Brown}, implying that
$\ga$ bounds a topological disk $\Delta$. But $\Delta$ is a smooth manifold   
since $\ga$ is a smooth bicollared submanifold of $\bbr^n$. 
If $n\geq 5$, $\Delta$ is then diffeomorphic to $D^n$
as a consequence of the $h$-cobordism theorem \cite[\S\,9, Propositions~A and~C]{Milnor}.
When $n=3$, we form consider the smooth manifold $\Sigma$ obtained by gluing 
$D^n$ to $\Delta$ using the diffeomorphism $c$.
Then $\Sigma$ is a smooth closed $3$-manifold which is homeomorphic to $S^3$. 
By the smoothing theorem \cite[Theorem~6.3]{Munkres}, $\Sigma$ is then diffeomorphic to $S^3$.
By a smooth ambient isotopy of $\Sigma\approx S^3$, the disk $D$ may be put in standard position 
(see~(b) in \secref{S.pfmainth}), 
which implies that its complementary $\Delta$ is diffeomorphic to $D^3$. 
\end{proof}

\begin{Remark}\label{R.schn=4}\rm
The smooth Schoenflies theorem is not known for $n=4$. 
By the end of the above proof of \proref{P.schoe}, it would be implied by
the smooth Poincar\'e conjecture in dimension~$4$.  \mancqfd 
\end{Remark}

\begin{proof}[Proof of \thref{mainThGen}]
The proof follows that of the main theorem, so we just describe below
the necessary adaptations.
As in \secref{S.sptri},
we define the \dfn{space of $n$-simplexes} (not reduced to a single point) $\tri^0(n)$ 
as the smooth manifold
$$
\tri^0(n)=\{(z_0,z_1,\dots,z_n)\in(\bbr^n)^{n+1}-\Delta \} \, ,
$$ 
where $\Delta$ is the diagonal in $(\bbr^n)^{n+1}$. 
The diagonal action of the translation group $\calg_1(n)\approx \bbr^n$ on $\tri^0(n)$ is smooth and proper,
with quotient $\tri^1(n)$ diffeomorphic to $\bbr^{n^2}-\{0\}$. A further quotient by the homotheties with positive ratio
provides $\tri(n) \approx S^{n^2-1}$.
% \signet
% Let $\calg_1(n)\approx \bbr^n$ be the group of translations of $\bbr^n$.
% Its diagonal action on $(\bbr^n)^{n+1}$  is smooth and proper and
% $\tri^1(n)=\tri^0(n)/\calg_1(n)$ is diffeomorphic to $\bbr^{n^2}-\{0\}$.
% Let $\calg_2(n)$ be the group of homothety of $\bbr^n$ with a positive ratio, isomorphic to $\bbr_{>0}$.
% Thus, $\tri(n) = \tri^1/\calg_2(n) \approx S^{n^2-1}$. 
% \signet
The spaces of flat simplexes $\ftri^0(n)$, $\ftri^1(n)$ and
$\ftri(n)$ are defined accordingly. Note that
$\ftri^1(n) = \delta^\mun(0)$, where $\delta\:\tri^1(n)\to\bbr$ is the smooth map
$$
\delta(z_1,\dots,z_n)= \det(z_1,\dots,z_n)
$$
the $z_i$'s being considered as column vectors of an $(n\times n)$-matrix. As $0$ is a regular value of $\delta$,
we get, as in \proref{P.tritrifl}, a diffeomorphism of manifold pairs
$$
(\tri(n),\ftri(n)) \approx (S^{n^2-1},\calt(n))
$$ 
where $\calt(n)$ is a codimension $1$ submanifold separating $S^{n^2-1}$ into two components. 

We now study the main map analogous to that of \secref{S.mainmap}.
Let $c\:S^{n-1}\to\bbr^n$ be a $\calc^1$-embedding parameterizing a Jordan sphere $\ga$.
Let 
$$
V=(S^{(n-1})^{(n+1)}-\Delta
$$ 
and consider the $\calc^1$-map $F^0\:V\to \tri^0(n)$ given by
$$
F^0(s_0,s_1,\dots,,s_n) = (c(s_0),c(s_1),\dots,c(s_n)) \, .
$$
The compositions with the quotient maps onto $\tri^1(n)$  and $\tri(n)$ give maps $F^1\:V\to \tri^1(n)$ and $F\:V\to\tri(n)$. 
As in \secref{S.compa}, we perform a boundary-compactification $\hat V$ of $V$ and extend the map $F$ 
into continuous map of pairs
\beq{gdefhatF}
\hat F \: (\hat V,\partial\hat V) \to (\tri,\ftri) \, .
\eeq
The definition of $\hat V$ goes as follows. 
The restriction of the tangent bundle $T((S^{n-1})^{n+1})$ over $\Delta$ admits the following description
$$
T((S^{n-1})^{n+1})_{|\Delta} = \{(s,w_0,\dots,w_n)\mid s\in S^{n-1}  \ , \ 
w_i\in\bbr^n  \hbox{ and } \langle w_i,s\rangle=0\}  \, .
$$
This bundle over $\Delta$ splits into a Whitney sum $T((S^{n-1})^{n+1})_{|\Delta} \approx T\Delta \oplus \caln\Delta$
of the tangent bundle $T\Delta$ (where all the $w_i$'s are equal) and the supplementary bundle
$$
\caln\Delta = \{(s,w_1,\dots,w_n)\mid s\in S^{n-1} \ , \ w_i\in\bbr^n \hbox{ and } \langle w_i,s\rangle=0\}  \, .
$$
Call $\caln^1\Delta$ the unitary bundle of $\caln\Delta$ (where $\sum_{i=1}^n|w_i|^2=1$).
For $s\in S^{n-1}$, consider the exponential map $\exp_s\:T_sS^{n-1}\to S^{n-1}$ for the standard metric on $S^{n-1}$.
The map 
$$
h\:\caln^1\Delta \times [0,1)\to (S^{n-1})^{n+1}
$$
given by 
$$
h((s,w_1,\dots,w_n),\lambda) = (s,\exp_s(\lambda w_1),\dots,\exp_s(\lambda w_n))
$$
parameterizes an open tubular neighborhoods of $\Delta$ in $(S^{n-1})^{n+1}$. The space $\hat V$ is 
the quotient of 
$\hat V = (\caln^1\Delta \times [0,1)) \,\,\dot\cup \,\, W$
% $$
% \hat V = \Big(\big\{\caln^1\Delta \times [0,1)\big\} \,\,\dot\cup \,\, W\Big) \, \Big/  \sim  \, ,
% $$
by the equivalence relation 
$$
(s,w_1,\dots,w_n),\lambda) \sim  h(s,w_1,\dots,w_n),\lambda) \ \hbox{ when } \ \lambda\neq 0  \, .
$$
Thus, $\hat V$ is a compact manifold with boundary $\partial\hat V\approx \caln^1\Delta$.
As 
$
\exp_s(\lambda w_i) - s = \exp_s(\lambda w_i) - \exp_s(0) = \lambda w_i + \pcirc(\lambda w_i) %\, ,
$,
one proves, as for \lemref{L.hatF}, that $F$ admits the continuous extension $\hat F$ of~\eqref{gdefhatF},
such that $\hat F \pcirc h (s,w_1,\dots,w_n),0)$ is represented by the flat simplex with vertices $T_sc(w_1),\dots,T_sc(w_n)$,
contained in the affine hyperplane tangent to $\ga$ at $c(s)$.

The proof of \thref{mainThGen} now proceeds as in \secref{S.pfmainth}. 
As $\dim \hat V= \dim \tri(n)$, the map $\hat F_c=F$ has a bidegree and \thref{mainThGen} comes from 
the equality $\bdg(\hat F_c)=(1,1)$, using \proref{P.bidegSurj}.
The case where $c(S^{n-1})=S^{n-1}$ is treated first, using that, on $S^{n-1}$,
every non-flat $n$-simplex occurs exactly once; by the Sard theorem, there must be regular values of $\hat F_c$
almost everywhere in $\tri(n)-\tri_{fl}(n)$, which proves that $\bdg(\hat F_c)=(1,1)$ by \proref{P.degree-sing}.                                                                                                                  
Point~(a) of \secref{S.pfmainth} is valid in any dimension, permitting us to assume that $c$ is of class $\calc^\infty$. 
By the smooth Schoenflies theorem (see \proref{P.schoe}),
there exists a $\calc^\infty$-embedding $\Theta\:D^n\to\bbr^n$ such that
$\Theta(S^{n-1})=\ga$. Replacing $c$ by $\Theta$ if necessary, we can assume that $c$ extends to an embedding of $D^n$.
Point~(b) of \secref{S.pfmainth} proves that such an embedding is isotopic to a diffeomorphism of $D^n$ and thus
$\bdg(\hat F_c)=(1,1)$. (Contrarily to the case $n=2$, the embedding $c$ itself may not extend to $D^n$,
since there are diffeomorphisms of $S^{n-1}$ which do not extend to diffeomorphisms of $D^n$). 
\end{proof}

%BIBLIOGRAPHY

\sk{4}\noindent {\small
Jean-Claude HAUSMANN\\
Math\'ematiques -- Universit\'e\\
B.P. 64, 
CH--1211 Geneva 4, Switzerland\\
jean-Claude.hausmann@unige.ch}

\end{document}